\documentclass[10pt]{amsart}
\usepackage{amsmath, amsfonts, amscd, latexsym, amsthm, amssymb}
\usepackage{hyperref}

\def \c{\mathbb{C}}
\def \z{\mathbb{Z}}
\def \r{\mathbb{R}}
\def \n{\mathbb{N}}
\def \p{\mathbb{P}}

\def \ol{\overline}

\def \L{\mathcal{L}}
\def \K{{\bf K}}

\def \S{\mathcal{S}}
\def \P{\mathcal{P}}
\def \V{\mathcal{V}}

\def \.{\cdot}

\def \Div{\textup{Div}}

\def \Pic{\textup{Pic}}

\def \Vol{\textup{Vol}}

\def \supp{\textup{supp}}

\def \ker{\textup{ker}}

\def \GL{\textup{GL}}

\def \Gr{\textup{Gr}}
\def \GL{\textup{GL}}
\def \SL{\textup{SL}}
\def \SO{\textup{SO}}
\def \SP{\textup{SP}}

\theoremstyle{plain}
\newtheorem{Th}{Theorem}[section]

\newtheorem{Prop}[Th]{Proposition}
\newtheorem{Cor}[Th]{Corollary}

\theoremstyle{definition}
\newtheorem{Ex}[Th]{Example}
\newtheorem{Def}[Th]{Definition}
\newtheorem{Rem}[Th]{Remark}

\begin{document}
\title{Newton polytopes for horospherical spaces}
\author{Kiumars Kaveh, A. G. Khovanskii}
\date{\today}

\maketitle


\begin{abstract}
A subgroup $H$ of a reductive group $G$ is horospherical if it contains a maximal
unipotent subgroup. We describe the Grothendieck semigroup of invariant subspaces of regular
functions on $G/H$ as a semigroup of convex polytopes.
From this we obtain a formula for the number of
solutions of a generic system of equations on $G/H$ in terms of mixed volume of polytopes.
This generalizes Bernstein-Kushnirenko theorem from toric geometry.\\
\end{abstract}



\section{Introduction}
Consider a commutative semigroup $S$. Two elements $a, b \in S$ are {\it analogous}
and written $a \sim b$ if there is $c \in S$ with $a+c = b+c$ (where we write the semigroup operation additively).
This relation is an equivalence relation and respects the addition.
The {\it Grothendieck semigroup $\Gr(S)$ of $S$} is the set of equivalence classes of $\sim$
together with the induced addition. The map which sends an element to its equivalence class
is a natural homomorphism $\rho: S \to \Gr(S)$. The semigroup $\Gr(S)$ has the
{\it cancelation property}, i.e. if $a, b, c \in \Gr(S)$ the equality $a+ c = b+c$ implies $a = b$. Moreover,
for any homomorphism $\varphi: S \to H$ where $H$ is a semigroup with cancelation, there exists a
unique homomorphism $\bar{\varphi}: \Gr(S) \to H$ such that $\varphi = \bar{\varphi} \circ \rho$.
In particular, under the homomorphism $\varphi$, analogous elements have the same image.
Any semigroup $H$ with cancelation naturally extends to a group, namely its {\it group of formal differences}.
It consists of pairs of elements from $H$ where two pairs $(a, b)$ and $(c,d)$ are equal
if $a+d = b+c$. The {\it Grothendieck group of a semigroup $S$} is the group of formal differences of $\Gr(S)$.

The Grothendieck semigroup of $S$ contains significant information about $S$ and often is more tractable
and simpler to describe than $S$ itself.

We will be interested in semigroups of subspaces of functions (as well as sections of
line bundles) which arise naturally in algebraic geometry. Let $X$ be an irreducible variety over $\c$ with the field of rational functions $\c(X)$. Consider the collection $\K(X)$ of all nonzero finite dimensional subspaces of $\c(X)$.
For $L_1, L_2 \in \K(X)$ let $L_1L_2$ denote the linear span of all $fg$, $f \in L_1$, $g \in L_2$.
With this product, $\K(X)$ is a commutative semigroup. One shows that
for each $L \in \K(X)$ there is a largest subspace $\ol{L}$ which is analogous to $L$ called
the {\it completion of $L$} (see \cite[Appendix 4]{Zariski} and \cite{Askold-Kiumars-MMJ}).

An interesting and important special case is the algebraic torus $X = (\c^*)^n$. The variety $X$ is a
multiplicative group and acts on itself by multiplication. For $\alpha = (a_1, \ldots, a_n) \in \z^n$ let
$x^\alpha = x_1^{a_1} \cdots x_n^{a_n}$ denote the corresponding Laurent monomial which is a
regular function on $X$. Let $A \subset \z^n$ be a finite subset and let $L_A$ denote the
Laurent monomials spanned by $x^\alpha$ for $\alpha \in A$. The correspondence $A \mapsto L_A$ gives an isomorphism
between the semigroup $\K_T(X)$ of invariant subspaces of regular functions on $X$ and the semigroup of
finite subsets of $\z^n$. One then shows that the Grothendieck semigroup of $\K_T(X)$ is
isomorphic to the semigroup of integral convex polytopes with Minkowski sum.
Moreover, {\it for a finite subset $A$, the completion of the subspace $L_A$ is the
subspace $L_{\ol{A}}$ where $\ol{A}$ is the set of all integral points in the convex hull $\Delta(A)$}.
From this key fact one can deduce the Bernstein-Kushnirenko theorem: let $A_1, \ldots, A_n \subset \z^n$ be
finite subsets. Then the number of solutions in $(\c^*)^n$ of a generic system $f_1(x) = \cdots = f_n(x)$ where
$f_i \in L_{A_i}$, is equal to $n! V(\Delta_1, \ldots, \Delta_n)$. Here, for each $i$, $\Delta_i$ is the convex hull of
$A_i$ and $V$ denotes the mixed volume of convex bodies in $\r^n$ (see \cite{Askold-sum-of-finite-sets} and \cite{Askold-Kiumars-Kazarnovskii}, also see \cite{Kushnirenko} and \cite{Bernstein} for the original papers where
this theorem appeared).

In this paper we consider a class of homogeneous spaces of reductive groups which have similar features
as the torus $(\c^*)^n$.
Let $G$ be a connected reductive algebraic group over $\c$. A subgroup
$H \subset G$ is called {\it horospherical} if it contains a maximal unipotent subgroup of $G$.
The homogeneous space $G/H$ is then called a {\it horospherical homogeneous space}. This class of varieties and their
partial compactifications have been studied in \cite{Popov-Vinberg}.

In the present paper we describe the semigroup of $G$-invariant subspaces of regular functions
on a horospherical homogeneous space $X = G/H$ (respectively its Grothendieck semigroup) in terms of a semigroup of finite subsets (respectively integral convex polytopes). Moreover, we obtain a similar description of the
completion of a finite dimensional $G$-invariant subspace of regular functions on $X$.
Finally we generalize the above to invariant linear systems on $X$.
(Theorem \ref{th-semi-gp-inv-subspace-G/P'}, Corollary \ref{cor-semi-gp-inv-lin-system-G/H}).

From these we obtain an analogue of the Bernstein-Kushnirenko theorem for the number of solutions
in $X$ of a system $f_1(x) = \cdots = f_n(x) = 0$ where each $f_i$ is a generic element in some
finite dimensional $G$-invariant subspaces of regular functions on $X$, in terms of mixed volume of
polytopes. More generally, we prove a similar statement for the $G$-invariant linear systems on $X$.
We give two answers for the number of solutions. Firstly, we represent the answer as the (mixed) integral of an
explicitly defined homogeneous polynomial over the so-called {\it moment polytope}
(Corollary \ref{cor-int-index-subspace}, Corollary \ref{cor-int-index-lin-system}).
Secondly, we construct larger polytopes over the moment polytopes
such that their (mixed) volume is equal to the above (mixed) integral
(Corollary \ref{cor-int-index-subspace-mixed-vol}, Corollary \ref{cor-int-index-lin-system-mixed-vol}).

This paper is one of a series of papers devoted to the general theory of convex bodies associated to
algebraic varieties.
In \cite{Askold-Kiumars-MMJ} we develop an intersection theory of finite dimensional subspaces of
rational functions. In \cite{Askold-Kiumars-Newton-Okounkov} we develop a general theory of Newton-Okounkov bodies
associated to algebraic varieties and more generally graded algebra.
Finally in \cite{Askold-Kiumars-reductive} we consider the case of general varieties with a reductive group action.

The results of the present paper are along the same lines as \cite{Askold-Kiumars-Kazarnovskii}.
In there we describe the Grothendieck semigroup of finite dimensional
representations of a reductive group $G$ with tensor product. From this we give a proof of Kazarnovskii's
theorem on the number of solutions in $G$ of a generic system of equations consisting of matrix elements of representations of $G$. Some of the background material in the present paper are taken from \cite{Askold-Kiumars-Kazarnovskii}.

Among the different generalizations of Bernstein-Kushnirenko theorem (e.g. in \cite{Brion}, \cite{Kazarnovskii} and \cite{Askold-Kiumars-Newton-Okounkov}) the generalization of Bernstein-Kushnirenko (for horospherical homogeneous spaces) in this paper is closets to the original Bernstein-Kushnirenko theorem. We expect that other formulae in toric geometry involving Newton polytopes also extend to the horospherical case.

We would like to emphasize that a main difference of our approach (with many other authors) in
computation of intersection indices is that we do not require the varieties to be complete or projective and hence do not need any compactification.

And about the organization of material: Part I is devoted to preliminaries
on subspaces of rational functions, linear systems and their intersection indices,
notions of mixed volume and mixed integral and finally
semigroup of finite subsets of $\r^n$ and its Grothendieck semigroup.
In Part II we cover the main results of the paper. Section \ref{sec-horo-subgroup} discusses classification of horospherical subgroups of $G$. Section \ref{sec-G/P'} describes the semigroup of invariant subspaces, its Grothendieck semigroup and gives a formula for the intersection index on quasi-homogeneous horospherical spaces in terms of moment
polytopes. Section \ref{sec-G/H} discusses similar material
for invariant linear systems on general horospherical spaces. Finally in Section \ref{sec-G-C} we construct larger polytopes over moment polytopes whose volumes give the intersection index. The last section considers the example of $G=\GL(n, \c)$.\\

\noindent{\bf Notation:} Throughout the paper we will use the following notation.
\begin{itemize}
\item[-] $G$ denotes
a connected reductive algebraic group over $\c$ with $\dim(G) = d$.
\item[-] $B$ denotes a Borel subgroup of $G$ and $T$ and $U$ the maximal torus
and maximal unipotent subgroups contained in $B$ respectively. We put $\dim(T) = r$.
\item[-] $W$ denotes the Weyl group of $(G,T)$.
\item[-] $\Lambda$ denotes the weight lattice of $G$
(that is, the character group of $T$), and $\Lambda^+$ is the subset
of dominant weights (for the choice of $B$). Put $\Lambda_\r = \Lambda \otimes_{\z} \r$.
Then the convex cone generated by $\Lambda^+$ in $\Lambda_\r$ is the
positive Weyl chamber $\Lambda^+_{\r}$.
\item[-] For a weight $\lambda \in
\Lambda$, the irreducible $G$-module corresponding to $\lambda$ will
be denoted by $V_\lambda$ and a highest weight vector in $V_\lambda$ will
be denoted by $v_\lambda$.
\item[-] For an algebraic group $K$, we denote the group of characters of $K$
(written additively) by $\mathfrak{X}(K)$.
\item[-] $P$ denotes a parabolic subgroup of $G$ and $P'$ its commutator subgroup.
\item[-] $H$ will denote a horospherical subgroup, i.e. a subgroup of $G$ containing a
maximal unipotent subgroup $U$.
\end{itemize}

\section{Part I: Preliminaries} \label{part-preliminaries}
\subsection{Intersection theory of finite dimensional subspaces and linear systems} \label{sec-int-index}
Let $X$ be a complex $n$-dimensional irreducible variety with $\c(X)$ its field of rational
functions. Consider the collection $\K(X)$ of all nonzero finite dimensional subspaces of
$\c(X)$. The {\it product} of two subspaces $L_1, L_2 \in \K(X)$ is the subspace spanned by all the
$fg$ where $f \in L_1$, $g \in L_2$. With this product $\K(X)$ is a commutative semigroup.

\begin{Def} \label{def-int-index}
The {\it intersection index} $[L_1, \ldots, L_n]$ is the number of solutions in $X$
of a generic system of equations $f_1 = \cdots = f_n = 0$ where $f_i \in L_i$, $1 \leq i \leq n$.
In counting the solutions, we
neglect the solutions $x$ at which all the functions in some space $L_i$
vanish as well as the solutions at which at least one function from some
space $L_i$ has a pole.
\end{Def}

One shows that the intersection index is well-defined (i.e. is independent of the choice of a
generic system) \cite{Askold-Kiumars-MMJ}. It is obvious that the intersection index is symmetric with respect
to permuting the subspaces $L_i$. Moreover, the intersection index is linear in each argument.
The linearity in first argument means:
\begin{equation} \label{equ-*}
[L'_1L''_1, L_2,\dots, L_n]= [L'_1, L_2,\dots, L_n] + [L''_1, L_2,\dots,L_n],
\end{equation}
for any $L'_1, L''_1, L_2, \ldots, L_n \in \K(X)$.
From (\ref{equ-*}) one sees that for a fixed $(n-2)$-tuple of subspaces
$L_2, \ldots, L_n \in \K(X)$, the map $\pi: \K(X) \to \r$ given by $\pi(L) = [L, L_2, \ldots, L_n]$
is a homomorphism from the semigroup $\K(X)$ to the additive group of integers. The existence of such
a homomorphism shows that the intersection index induces an intersection index on
$\Gr(\K(X))$, i.e. {\it the intersection index $[L_1, \ldots, L_n]$ remains invariant if we substitute each
$L_i$ with an analogous subspace $\tilde{L}_i$.}

One can describe the relation of analogous subspaces in a different way as follows (see \cite{Askold-Kiumars-MMJ}).
A rational function $f \in \c(X)$ is called {\it integral over the subspace $L$} if
it satisfies an equation $$f^m+a_1 f^{m-1} + \dots a_0 =0$$ with $m>0$ and $a_i \in L^i$, $1 \leq i \leq m$.
The collection of all the rational functions integral over $L$ forms a finite dimensional subspace
$\overline{L}$ called the {\it completion of $L$}.

\begin{Prop} \label{prop-1}
1) Two subspaces $L_1, L_2 \in \K(X)$ are analogous if and only if $\overline{L}_1 = \overline{L}_2$.
2) For any $L \in \K(X)$, the completion $\overline{L}$ belongs to $\K(X)$ and
is analogous to $L$.
3) Moreover, the completion $\overline{L}$ contains all the subspaces $M \in \K(X)$ analogous to $L$.
\end{Prop}

For $L \in \K(X)$ define the Hilbert function $H_L$ by $H_L(k) = \dim (\overline{L^k})$.
The following theorem provides a way to compute the self-intersection index of a subspace $L$
(see \cite[Part II]{Askold-Kiumars-Newton-Okounkov}):
\begin{Th} \label{th-intersec-index-Hilbert-function}
For any $L \in \K(X)$, the limit $$a(L) = \lim_{k \to \infty} H_L(k)/k^n$$ exists, and
the self-intersection index $[L, \ldots, L]$ is equal to $n! a(L)$.
\end{Th}
The proof is based on the Hilbert theorem on the dimension and degree of a subvariety of the projective space.\\

A linear system on $X$ is a family of effective divisors of the form $D + (f)$ where $D$ is an
effective divisor on $X$ and $f$ lies in a finite dimensional subspace $L \subset \c(X)$. In this
section we consider the intersection index of linear systems.
Let us assume that $D$ is a Cartier divisor and let $\L$ be the line bundle associated to $D$.
Any element $D + (f)$ determines a section of the line bundle $\L$ up to multiplication by a
regular nowhere zero function, i.e. an element of $\c[X]^*$. Thus a linear system determines a
subspace of holomorphic sections of $\L$ up to multiplication by a function in $\c[X]^*$.

Conversely a finite dimensional subspace $E$ of holomorphic sections $H^0(X, \L)$ determines a
linear system of divisors $\{ \Div(s) \mid 0 \neq s \in E \}.$
By abuse of terminology we will refer to $(E, \L)$ (or simply $E$) as a linear system on $X$.
Fix a nonzero section $t \in E$. Then every other section $s \in E$ can be written as
$s = f_s t$ for a unique $f_s \in \c(X)$. The map $s \mapsto f_s$ identifies $E$ with the
subspace of rational functions $\{ f_s \mid s \in E \}$.

Let $(E_1, \L_1)$, $(E_2, \L_2)$ be two linear systems on $X$. There is a tensor product map
$H^0(X, \L_1) \otimes H^0(X, \L_2) \to H^0(X, \L_1 \otimes \L_2)$. Let $E_1 E_2$ denote
the span of all the products $f_1f_2 \in H^0(X, \L_1 \otimes \L_2)$ for $f_1 \in E_1$,
$f_2 \in E_2$. We call $(E_1E_2, \L_1 \otimes \L_2)$ the {\it product of two linear systems
$(E_1, \L_1)$, $(E_2, \L_2)$}. With this product the collection $\tilde{\K}(X)$ of all
the linear systems on $X$ is a commutative semigroup.

Again fix a nonzero section $t \in E$ and let $L = \{f_s \mid s \in E \}$ be the corresponding
subspace of rational functions. Define the {\it completion of the linear system $E$} to be the
subspace $\ol{E} = \{ ft \mid f \in \ol{L}\}$ where $\ol{L}$ is the completion of the subspace
$L$ (as defined above). If $X$ is normal, one verifies that $\ol{E}$ still consists of holomorphic sections i.e.
$\ol{E} \subset H^0(X, \L)$.
One also verifies that for any rational function $h$ we have $\ol{hL} = h \ol{L}$, from which it follows
that $\ol{E}$ is well-define, i.e. is independent of the choice of the section $t$.

A linear system is said to have {\it no base locus} if the intersection of the supports of the
divisors $D + (f)$, $\forall f \in L$, is empty. In other words, if $E \subset \L$ is a subspace of
holomorphic sections representing a linear system then $E$ has no base locus if for any $x \in X$ there
is $s \in E$ with $s(x) \neq 0$.

\begin{Def}[Intersection index of linear systems] \label{def-int-index-lin-system}
Let $\L_1, \ldots, \L_n$ be line bundles on $X$ with linear systems
$E_i \subset H^0(X, \L_i)$ for $i=1,\ldots, n$ with no base locus. The {\it intersection index}
$[E_1, \ldots, E_n]$ is the number of points in $D_1 \cap \cdots \cap D_n$ where
$D_i$ is a generic divisor in the linear system $E_i$, i.e. $D_i = \Div(s_i)$ where
$0 \neq s_i$ is a generic element of $E_i$. For each $i$, fix a section $t_i \in E_i$ and let
$L_i \in \c(X)$ be the subspace associated to $E_i$ and $t_i$. One sees that
$[E_1, \ldots, E_n]$ is in fact equal to the intersection index $[L_1, \ldots, L_n]$ of
subspaces of rational functions and hence is well-defined.
\end{Def}

The intersection index of linear systems enjoys properties similar to the intersection index of
subspaces:
\begin{enumerate}
\item The intersection index is symmetric with respect to permuting the arguments.
\item The intersection index is multi-linear with respect to the product of linear systems.
\item The intersection index $[E_1, \ldots, E_n]$ does not change if we replace any of the $E_i$ with
an analogous linear system $\tilde{E}_i$ (in particular with the completion $\ol{E_i}$).
\item As for subspaces of rational functions, for a linear system $E$ on $X$ let us define the
{\it Hilbert function}
by $H_E(k) = \dim(\ol{E^k})$. Then the limit $$a(E) = \lim_{k \to \infty} H_E(k)/k^n$$ exists, and
the self-intersection index $[E, \ldots, E]$ is equal to $n! a(E)$.
\end{enumerate}

\subsection{Mixed volume and mixed integral} \label{sec-mixed-vol}
A function $F: \V \to \r$ on a (possibly infinite dimensional) vector space $\V$ is called
a homogeneous polynomial of degree $k$ if its restriction to any finite dimensional subspace of $\V$
is a homogeneous polynomial of degree $k$. (For any $k$, the constant zero function is a homogeneous
polynomial of degree $k$.)

\begin{Def}
To a symmetric multi-linear function $B(v_1, \ldots, v_k)$, $v_i \in \V$ one corresponds a homogeneous polynomial $P$ of degree $k$ on $\V$ defined by $P(v) = B(v, \ldots, v)$. We say that the symmetric form $B$ is a
{\it polarization of the homogeneous polynomial $P$}.
\end{Def}

If $F$ is a homogeneous polynomial of degree $k$, then its derivative $F'_v(x)$ in the direction of a
vector $v$ is linear in $v$ and homogeneous of degree $k-1$ in $x$. Let $v_1, \ldots, v_k$ be
a $k$-tuple of vectors. For each $x$, the $k$-th derivative $F^{(k)}_{v_1, \ldots, v_k}(x)$
is a symmetric multi-linear function in the $v_i$. One easily verifies the following:

\begin{Prop}
Any homogeneous polynomial of degree $k$ has a unique {\it polarization} $B$ defined by the
formula: $$B(v_1,\dots,v_k)= (1/k!) F^{(k)}_{v_1,\dots,v_k}.$$
\end{Prop}

A compact convex subset of $\r^n$ is called a {\it convex body}.
Consider the collection of convex bodies in $\r^n$. There are two
operations of Minkowski sum and multiplication by a non-negative
scalar on convex bodies. The collection of convex bodies with
Minkowski sum is a semigroup with cancelation. The
multiplication by a non-negative scalar is associative and
distributive with respect to the Minkowski sum. These
properties allow us to extend the collection of convex bodies to the
(infinite dimensional) vector space $\V$ of {\it virtual convex
bodies} consisting of formal differences of convex bodies (see
\cite{Burago-Zalgaller}).

Let $d\mu = dx_1 \cdots dx_n$ be the standard  Euclidean measure in $\r^n$.
For each convex body $\Delta \subset \r^n$ let $\Vol(\Delta) = \int_\Delta d\mu$ be its volume.
The following statement is well-known:

\begin{Prop}
The function $\Vol$ has a unique extension to the vector space $\V$ of virtual convex bodies as a
homogeneous polynomial of degree $n$.
\end{Prop}

\begin{Def}
The {\it mixed volume} $V(\Delta_1, \ldots, \Delta_n)$ of the convex bodies $\Delta_i$
is the value of the polarization of the volume polynomial $\Vol$ at $(\Delta_1, \ldots, \Delta_n)$.
\end{Def}

Fix a homogeneous polynomial $F$ of degree $p$ in $\r^n$. Let
$IF(\Delta) = \int_\Delta F d\mu$ denote the integral of $F$ on $\Delta$.
One has the following (see for example \cite{Kh-P}):

\begin{Prop}
The function $IF$ has a unique extension to the vector space $\V$ of virtual convex bodies as
a homogeneous polynomial of degree $n+p$.
\end{Prop}

\begin{Def}
The mixed integral $IF(\Delta_1, \ldots, \Delta_{n+p})$ of a homogeneous polynomial $F$ over the bodies
$\Delta_1, \ldots, \Delta_{n+p}$ is the value of the polarization of the
polynomial $IF$ at the bodies $\Delta_1, \ldots, \Delta_{n+p}$.
\end{Def}
From definition, the mixed integral of the constant polynomial $F \equiv 1$ is the mixed volume.

More generally we can consider the mixed volume and mixed integral for convex bodies in $\r^n$ which
are parallel to a fixed subspace of $\r^n$. Fix a subspace $\Pi \subset \r^n$ with $\dim(\Pi) = m$.
Consider the collection of convex bodies which are parallel to $\Pi$, i.e.
lie in a translate $a+\Pi$ of $\Pi$ for some $a \in \r^n$. This collection is closed under addition and
multiplication by nonnegative scalars. Let $\V(\Pi)$ denote the subspace of all virtual convex bodies $\V$ spanned by the convex bodies parallel to $\Pi$. Fix a Lebesgue measure on $\Pi$ and equip each translate of $\Pi$ with a
Lebesgue measure by shifting the measure on $\Pi$. We denote all these measures by $d\gamma$. Let
$\Delta \subset a+\Pi$ be a convex body parallel to $\Pi$. The map $$\Delta \mapsto \Vol_\Pi(\Delta)$$ is a
homogeneous polynomial of degree $m$ on the vector space $\V(\Pi)$ where $\Vol_\Pi$ is the volume with respect to
the Lebesgue measure $d\gamma$. We will denote the polarization
of $\Vol_\Pi$ on $\V(\Pi)$ by $V_\Pi$ and call it the {\it mixed volume of convex bodies parallel to $\Pi$}.

Similarly, let $F$ be a homogeneous polynomial on $\r^n$ of degree $d$. Then the map
$$\Delta \mapsto \int_\Delta F d\gamma$$
is a homogeneous polynomial on $\V(\Pi)$. We will denote the polarization of this by
$IF_\Pi$. It is a $(m+d)$-linear function on $\V(\Pi)$.

\subsection{Semigroup of finite sets with respect to addition} \label{sec-finite-sets}
There is an addition operation on the collection of subsets of $\r^n$. The sum of two sets $A$ and $B$ is the set
$A+B = \{a+b \mid a \in A,~ b \in B\}$. One verifies that the sum of two convex bodies (respectively convex
integral polytopes) is again a convex body (respectively a convex integral polytope).
This is the well-known Minkowski sum of convex bodies. Consider the following:
\begin{itemize}
\item[-] $\S$, the semigroup of all finite subsets of $\z^n$ with the addition of subsets.
\item[-] $\P$, the semigroup of all convex integral polytopes with the Minkowski sum.
\end{itemize}

\begin{Prop} \label{prop-7}
The semigroup $\P$ has cancelation property.
\end{Prop}
Proposition \ref{prop-7} follows from the more general fact that the semigroup of convex bodies with respect to
the Minkowski sum has cancelation property.
The next statement is easy to verify:
\begin{Prop} \label{prop-8}
The map which associates to a finite nonempty set $A \subset \z^n$ its convex hull $\Delta(A)$,
is a homomorphism of semigroups from $\S$ to $\P$.
\end{Prop}

For an integral convex polytope $\Delta \in \P$ let $\Delta_\z \in \S$ denote the finite set
of integral points in $\Delta$, i.e. $\Delta_\z = \Delta \cap \z^n$.
It is not hard to verify the following (see \cite{Askold-sum-of-finite-sets}):
\begin{Prop} \label{prop-9}
For any nonempty subset $A \subset \z^n$ we have:
$$A + n\Delta(A)_\z = (n+1)\Delta(A)_\z = \Delta(A)_\z + n\Delta(A)_\z.$$
\end{Prop}

We then have the following description for the Grothendieck semigroup of $\S$.
\begin{Th} \label{th-finite-sets-convex-polytope-semigps}
The Grothendieck semigroup of $\S$ is isomorphic to $\P$. The homomorphism $\rho: \S \to \P$ is
given by $\rho(A) = \Delta(A)$.
\end{Th}
\begin{proof}
From Propositions \ref{prop-7} and \ref{prop-8} it follows that if $A \sim B$ then $\Delta(A) = \Delta(B)$.
Conversely, from Proposition \ref{prop-9} we know that $A$ and $\Delta(A)_\z$ are analogous.
By definition if $\Delta(A) = \Delta(B)$ then $\Delta(A)_\z = \Delta(B)_\z$ and hence $A$ and $B$ are analogous.
\end{proof}

\section{Part II: Horospherical homogeneous spaces} \label{part-main}
\subsection{Horospherical subgroups} \label{sec-horo-subgroup}
\begin{Def}[Horospherical subgroup]
A subgroup $H \subset G$ is called {\it horospherical} if it contains a maximal
unipotent subgroup. The corresponding homogeneous space $G/H$ is called a {\it
horospherical homogeneous space}.
\end{Def}

The horospherical spaces (respectively their partial compactifications called $S$-varieties) have
features similar to algebraic torus (respectively toric varieties).

The next theorem gives a description of the horospherical subgroups of $G$.
Recall that a subgroup $P$ of $G$ is parabolic if it contains a Borel subgroup. 

\begin{Th} \label{th-P-H}
Let $H$ be a horospherical subgroup of $G$. Then
there exists a parabolic subgroup $P$ of $G$ such that
$P' \subset H \subset P$, where $P'$ denotes the commutator subgroup of $P$.
Conversely, any closed subgroup $H$ with $P' \subset H \subset P$ is horospherical.
\end{Th}
\begin{proof}
Let $H$ be a horospherical subgroup containing a maximal unipotent subgroup $U$. By Chevalley's theorem
we can find a finite dimensional $G$-module $V$ and a vector $0 \neq v \in V$ such that $H$ is the stabilizer of
the point $[v]$ in the projective space $\p(V)$. Since $U \subset H$ and $U$ has no characters we see that
$v$ is fixed by $U$ and hence should be a sum of highest weight vectors. Let us write $v = \sum_{i=1}^s v_i$
where each $v_i$ is a highest weight vector of some weight $\lambda_i$. The Borel $B$ stabilizes the
point $x = ([v_1], \ldots, [v_s]) \in \prod_{i=1}^s \p(V_{\lambda_i})$ and hence the stabilizer subgroup $P$ of 
$x$ is a parabolic subgroup. Now since $H$ also stabilizes $x$ we have $H \subset P$ as required. 
Finally, the characters $\lambda_i$ restrict trivially to $P'$ and
thus $P'$ fixes the point $[v] \in \p(\bigoplus_{i=1}^s V_{\lambda_i})$, which proves that $P' \subset H$.
To prove the converse statement we need to show that $U \subset P'$. But $U$ is the commutator of $B$ and $B \subset P$. This finishes the proof.
\end{proof}

\subsection{Semigroup of invariant subspaces of $\c[G/P']$} \label{sec-G/P'}
Fix a Borel subgroup $B$ and let $U$ be its maximal unipotent subgroup.
One knows that there is a one-to-one correspondence between the parabolic subgroups containing
$B$ and the faces of the positive Weyl chamber $\Lambda^+_\r$.
Let $\sigma$ be a face of the positive Weyl chamber $\Lambda_\r^+$. Let $\sigma_\r$
denote the linear span of the cone $\sigma$. Also let $\Lambda_\sigma = \Lambda \cap \sigma_\r$ denote the lattice of weights
lying on $\sigma_\r$ and let $\Lambda_\sigma^+ = \Lambda^+ \cap \sigma$ be the semigroup of dominant weights lying on
the face $\sigma$. Let $P$ be the parabolic subgroup containing $B$ which corresponds to $\sigma$ and $P'$ its  commutator subgroup.
The inclusion $i: B \hookrightarrow P$ induces a map $i^*: \mathfrak{X}(P) \to \mathfrak{X}(B)=\Lambda$.
The following is well-known:
\begin{Prop} \label{prop-char-lattice-parabolic}
The map $i^*$ is an inclusion, i.e. each character of $P$ is determined by its restriction to $B$ (equivalently $T$).
Moreover, the image of $i^*$ coincides with the lattice $\Lambda_\sigma$, i.e. the characters which lie on the
linear span of the face $\sigma$. In particular, the rank of the lattice $\mathfrak{X}(P)$ is equal to the
dimension of the face $\sigma$.
\end{Prop}
We will identify the character group $\mathfrak{X}(P)$ with $\Lambda_\sigma$.

Consider the quotient group $S = P/P'$. By definition of $P'$, $S$ is an abelian algebraic group.
The natural projection map $\pi: P \to S$ gives a map $\pi^*: \mathfrak{X}(S) \to \mathfrak{X}(P)$.
\begin{Prop} \label{prop-P/P'}
The group $S$ is a torus of dimension equal to the dimension of the face $\sigma$.
Moreover, the map $\pi^*$ gives an isomorphism
between the character lattice of $S$ and the lattice $\Lambda_\sigma$
\end{Prop}
We will also identify $\mathfrak{X}(S)$ with $\Lambda_\sigma$.

Now consider the homogeneous space $X = G/P'$. As $P'$ is a normal subgroup of $P$, the group $P$
and hence $S = P/P'$ act on $X$ from right. Also $G$ acts on $X$ from left and the two actions
commute.

The following theorem is well-known and plays an important role for us.
\begin{Th} \label{th-G/P'}
\begin{enumerate}
\item The variety $X$ is a quasi-affine variety.
\item The ring of regular functions $\c[X]$ decomposes as:
$$\c[X] = \bigoplus_{\lambda \in \Lambda_\sigma^+} W_\lambda,$$
where $W_\lambda$ denotes the $\lambda$-eigenspace for the action of $S$.
Moreover, as a $G$-module $W_\lambda$ is isomorphic to the dual representation $V_\lambda^*$.
\item For any two dominant weights $\lambda, \mu \in \Lambda_\sigma^+$ we have
$$W_\lambda W_\mu \subset W_{\lambda + \mu}.$$
\end{enumerate}
\end{Th}

\begin{Def} \label{def-supp-C[X]}
\begin{enumerate}
\item Let $0 \neq f \in \c[X]$. Then we can write $f = \sum_{\lambda \in \Lambda_\sigma^+} f_\lambda$
where $f_\lambda \in W_\lambda \cong V_\lambda^*$. The {\it support of $f$} is the collection
$\supp(f)$ of all the dominant weights $\lambda$ where $f_\lambda \neq 0$. We define the support of
the $0$ function to be the empty set.

\item  Let $L \subset \c[X]$ be a $G$-invariant subspace of regular functions on $X$
(which is also automatically invariant under the right $S$-action). The {\it support of $L$}
is the collection $\supp(L)$ of dominant weights such that
$$L = \bigoplus_{\lambda \in \supp(L)} W_\lambda.$$
In other words, $\supp(L)$ is the union of all the $\supp(f)$ for $f \in L$.

\item Let $A \subset \Lambda_\sigma^+$ be a finite set. Put
$$L_A = \bigoplus_{\lambda \in A} W_\lambda.$$
In other words, $L_A$ is the collection of
all the $f \in \c[X]$ with $\supp(f) \subset A$.
By Theorem \ref{th-G/P'}, $L_A$ is a finite dimensional $(G \times S)$-invariant subspace of $\c[X]$.
\end{enumerate}
\end{Def}

\begin{Def} \label{def-moment-polytope}
For a $G$-invariant subspace $L \subset \c[X]$ we denote the convex hull of $\supp(L)$ by
$\Delta(L)$ and call it the {\it moment polytope of $L$}.
\end{Def}

\begin{Def} \label{def-semigp-inv-subspaces-C[X]}
We denote the collection of all the finite dimensional subspaces of $\c[X]$
which are invariant under the left $G$ action by $\K_G(X)$.
\end{Def}
The set $\K_G(X)$ is a semigroup under the
product of subspaces. Moreover, if $L \in \K_G(X)$ is a $G$-invariant subspace then
its integral closure $\overline{L}$ is also $G$-invariant i.e. $\overline{L} \in \K_G(X)$.

The next proposition follows immediately from Theorem \ref{th-G/P'}.
\begin{Prop} \label{prop-support-additive-C[X]}
\begin{enumerate}
\item Let $L_1, L_2 \in \K_G(X)$ be two $G$-invariant subspaces. Then
$$\supp(L_1 L_2) = \supp(L_1) + \supp(L_2),$$ and hence
$$\Delta(L_1L_2) = \Delta(L_1) + \Delta(L_2).$$
\item Let $A_1, A_2 \subset \Lambda_\sigma^+$ be finite subsets. Then
$$L_{A_1+A_2} = L_{A_1} L_{A_2}.$$
\end{enumerate}
\end{Prop}

Let $\mathcal{S}(\Lambda_\sigma^+)$ denote the semigroup of all finite subsets of
$\Lambda_\sigma^+$ together with the operation of addition of subsets.
Also let $\mathcal{P}(\Lambda_\sigma^+)$ be the semigroup of all convex polytopes in $\sigma$
with vertices in $\Lambda_\sigma^+$ together with the Minkowski sum of convex sets. By Theorem
\ref{th-finite-sets-convex-polytope-semigps}, the map $A \mapsto \Delta(A)$, the convex hull of $A$,
gives an isomorphism between the Grothendieck semigroup of $\mathcal{S}(\Lambda_\sigma^+)$ and the semigroup
$\mathcal{P}(\Lambda_\sigma^+)$.

The following theorem is a corollary of Proposition \ref{prop-support-additive-C[X]}
\begin{Th}[Semigroup of invariant subspaces] \label{th-semi-gp-inv-subspace-G/P'}
\begin{enumerate}
\item The map $L \mapsto \supp(L)$ gives an isomorphism of the semigroup $\K_G(X)$ of invariant subspaces
of regular functions and the semigroup $\mathcal{S}(\Lambda_\sigma^+)$ of finite subsets of $\Lambda_\sigma^+$.

\item The map $L \mapsto \Delta(L)$ gives an isomorphism of the Grothendieck semigroup of $\K_G(X)$ and the
semigroup $\mathcal{P}(\Lambda_\sigma^+)$ of convex lattice polytopes in $\sigma$.

\item If $L \in \K_G(X)$, the completion $\ol{L}$ is given by
$$\ol{L} = \bigoplus_{\lambda \in \Delta(L) \cap \Lambda_\sigma^+} W_\lambda.$$ Thus under the isomorphism
in the part (1), $\ol{L}$ corresponds to the finite set of all the dominant weights in the moment polytope $\Delta(L)$.
\end{enumerate}
\end{Th}


According to the {\it Weyl dimension formula} the dimension of an irreducible representation
$V_\lambda$ is equal to $F(\lambda)$,
where $F$ is a  polynomial on $\r^r$ of degree $(d-r)/2$ defined explicitly in terms of data associated to the Weyl group $W$. We call $F$ the {\it Weyl polynomial of $W$}. Let $F_\sigma$ denote the restriction of $F$ to the linear span $\sigma_\r$ of the face $\sigma$, and let $\phi_\sigma$ denote the homogeneous component of $F_\sigma$ of highest degree.

\begin{Cor}[Intersection index] \label{cor-int-index-subspace}
Let $L_1, \ldots, L_p \in \K_G(X)$ be $G$-invariant subspaces where $p = \dim(X)$. For
each $i$, let $\Delta_i = \Delta(L_i)$ be the moment polytope of the subspace $L_i$. We have
$$[L_1, \ldots, L_p] = p! I\phi_\sigma(\Delta_1, \ldots, \Delta_p),$$ where
$I\phi_\sigma$ is the mixed integral (Section \ref{sec-mixed-vol}).
\end{Cor}

\begin{Rem}
Note that each $L_i$ is a subspace of regular functions and hence elements of the $L_i$ do not have poles.
Also as each $L_i$ is $G$-invariant, the base locus of $L_i$ (i.e. where all the elements of $L_i$ vanish)
is $G$-invariant. But $G$ acts transitively on $X$ and $L_i \neq \{0\}$, it follows that $L_i$ has
no base locus. That is, the intersection index $[L_1, \ldots, L_p]$ counts the number of solutions of a
generic system in the whole $X$ (see Definition \ref{def-int-index}).
\end{Rem}

\subsection{Semigroup of $G$-invariant linear systems on $G/H$} \label{sec-G/H}
Fix a Borel subgroup $B$ with a maximal unipotent subgroup $U$.
Let $H$ be a subgroup of $G$ containing $U$ (i.e. $H$ is a horospherical subgroup).
In this section we consider the horospherical homogeneous space $Y = G/H$.
We describe the semigroup of invariant linear systems on $Y$ and its Grothendieck semigroup as well as
the intersection index of such linear systems.

From Theorem \ref{th-P-H} we know that there exists a parabolic subgroup $P$ containing $B$
such that $P' \subset H \subset P$. Let $\sigma$ be the face of positive Weyl chamber corresponding to
the parabolic subgroup $P$.

The inclusion $i: H \hookrightarrow P$ induces a
restriction map $i^*: \mathfrak{X}(P) \to \mathfrak{X}(H)$. As in Proposition
\ref{prop-char-lattice-parabolic} identify $\mathfrak{X}(P)$ with the
lattice $\Lambda_\sigma$ and let $\Lambda(H) \subset \Lambda_\sigma$ be the
kernel of the map $i^*$. Alternatively, $H/P'$ is a subgroup of the torus $S = P/P'$ and
$\Lambda(H)$ can be viewed as the kernel of the restriction map $\mathfrak{X}(S) \to \mathfrak{X}(H/P')$.
Also let $\Lambda_\r(H) = \Lambda(H) \otimes \r$ denote the linear span of the lattice $\Lambda(H)$.

Let $D$ be a divisor on $X$ and
$\{D + (f) \mid f \in L\}$ be a family of equivalent divisors (i.e a linear system) on $X$ where $L$ is a finite dimensional
subspace of rational functions. Let us assume that the family is invariant under the action of $G$, i.e. for
each $g \in G$ and $f \in L$ we have $(g \cdot D) + (g \cdot f) = D + (h)$ for some $h \in L$.
If we assume that the only regular nowhere zero functions on $X$ are constants then
the principal divisor $(h)$ determines $h$ up to a constant. One verifies that $g: f \mapsto h$ gives a
projective representation of the group $G$ in the projective space $\p(L)$.
For simplicity let us assume that this lifts to a linear representation of $G$ on $L$. Thinking of
a linear system as a subspace of sections of the line bundle $\L$ (associated to the divisor $D$) we have
the following definition.
(Recall that a $G$-linearized line bundle $\L$ on $Y$ is a line bundle $\L$ with an action of $G$ on $\L$ extending its
action on $Y$ such that for any $x \in X$ the action of $g \in G$ maps the fibre $\L_x$ linearly to the fibre
$\L_{g \cdot x}$.)
\begin{Def}[$G$-invariant linear system] \label{def-G-lin-system}
We call a $G$-invariant finite dimensional subspace
$E \subset H^0(Y, \L)$, a {\it $G$-invariant linear system on $Y$}.
We denote the collection of all such pairs $(E, \L)$ (up to isomorphism)
by $\tilde{\K}_G(Y)$.
\end{Def}

The product of two $G$-invariant linear systems is again invariant and hence
the set $\tilde{\K}_G(Y)$ is a semigroup with respect to the product of linear systems.
Moreover, if $(E, \L) \in \tilde{\K}_G(Y)$ then
the completion $\overline{E}$ is $G$-invariant i.e. $(\overline{E}, \L) \in \tilde{\K}_G(Y)$.

\begin{Def}[Support of an invariant linear system] \label{def-support-moment-polytope-G-lin-system}
Let $E$ be a $G$-invariant linear system on $Y$. The {\it support of
$E$} is the set $\supp(E)$ of all dominant weights $\lambda \in \Lambda_\sigma^+$ for which
$V_\lambda^*$ appears in the decomposition of $E$ into irreducible $G$-modules.
\end{Def}

\begin{Def}[Newton polytope of a linear system]
Let $E$ be a $G$-invariant linear system.
We call the convex hull of $\supp(E)$, the {\it moment polytope of $E$} and denote it by $\Delta(E)$.
\end{Def}

Consider the natural projection $\pi: X=G/P' \to Y=G/H$. We would like to look at the pull-back $\pi^*(E)$ of a
$G$-invariant linear system $E$ on $Y$ to $X$.

Without loss of generality we can assume that every hypersurface in $G$ is given by an equation, that is,
$\Pic(G) = \{0\}$. In fact, by a theorem of Popov (see \cite{Popov}) for any connected linear algebraic group $G$
there exists a central isogeny $\pi: \tilde{G} \to G$ such that $\Pic(\tilde{G}) = \{0\}$. (That is, $\pi$ is an
algebraic homomorphism such that $\ker(\pi)$ is finite and lies in the center of $G$.)
Now if $U$ is a maximal unipotent subgroup of $G$ then $\pi^{-1}(U)$ is also a maximal unipotent subgroup of $\tilde{G}$. Thus
replacing $G$ with $\tilde{G}$ we can assume that the Picard group of $G$ is trivial.
The next theorem is an immediate corollary of another result in \cite{Popov}:

\begin{Th}
The character group $\mathfrak{X}(P')$ is trivial and hence any $G$-linearized
line bundle on $X=G/P'$ is $G$-equivariantly isomorphic to the trivial bundle where $G$ acts on each fibre trivially.
\end{Th}

Let $\L$ be a $G$-linearized line bundle on $Y$. Then by the above theorem the pull-back line bundle
$\pi^*(\L)$ is trivial. Thus we can identify the space of sections $H^0(X, \pi^*(\L))$ with the
ring of regular functions $\c[X]$. Note that since $\pi: X \to Y$ is surjective, the map
$\pi^*: H^0(Y, \L) \to H^0(X, \pi^*(\L))\cong \c[X]$ is one-to-one.
The map $\pi^*$ then identifies
a $G$-invariant linear system $E$ on $Y$ with a $G$-invariant subspace $L(E) \subset \c[X]$.
It is clear that $\supp(E) = \supp(L(E))$.

\begin{Th} \label{th-E-L(E)}
The map $E \mapsto L(E)$ gives a one-to-one correspondence between the collection of $G$-invariant linear systems on $Y$ (up to isomorphism) and the collection of finite dimensional subspaces of $\c[X]$ which are invariant under the left action of $G$ and lie in an eigenspace for the right action of $H$.
\end{Th}
\begin{proof}
Each character $\gamma \in \mathfrak{X}(H)$ gives a $G$-linearized line bundle $\L_\gamma$ on $Y = G/H$ defined by
$\L_\gamma = (G \times \c)/H$ where $h \in H$ acts on $(g, x) \in G \times \c$ by $h \cdot (g, x) = (gh^{-1}, \gamma(h)x)$.
The fibration $\L_\gamma \to G/H$ is given by the projection on first factor.
From definition, each holomorphic section of $\L_\gamma$ corresponds to a section of the
trivial bundle over $G$ which is invariant under the above action of $H$.
It follows that the holomorphic sections of $\L_\gamma$ are in one-to-one correspondence with regular functions in $\c[G]$ which are $\gamma$-eigenfunctions for the right action of $H$.
By a theorem of Popov (\cite{Popov}) the correspondence $\gamma \mapsto \L_\gamma$ is an isomorphism of
$\mathfrak{X}(H)$ and $\Pic_G(G/H)$, the group of $G$-linearized line bundles on $G/H$ (with tensor product).
Thus if $E \subset H^0(Y, \L_\gamma)$ is an invariant linear system where
$\L$ is a $G$-linearized line bundle, then for some character $\gamma \in \mathfrak{X}(H)$ we have $\L = \L_\gamma$ and
$E$ can be identified with a (left) $G$-invariant subspace of $\gamma$-eigenspace of $\c[G]$ for the right action of $H$. Also as $P'$ has no characters, each $\gamma$-eigenfunction is $P'$-invariant and hence belongs to $\c[G/P']$. This proves the proposition.
\end{proof}

The next proposition describes the support of the subspace $L(E)$ associated to an invariant system $E$.
\begin{Prop} \label{prop-supp-parallel-Lambda-H}
\begin{enumerate}
\item Let $\gamma \in \mathfrak{X}(H)$ be a character of $H$.
Let $L$ be a (left) $G$-invariant subspace of $\c[G/P']$ consisting of $\gamma$-eigenfunctions of (right) action
 of $H$. Then $\supp(L)$ is contained in a coset of $\Lambda(H)$, i.e. for any $\lambda, \mu \in \supp(L)$
we have $\lambda - \mu \in \Lambda(H)$. In particular, the smallest affine space spanned by $\supp(L)$
is parallel to the linear space $\Lambda_\r(H)$.
\item Conversely, let $A \subset \Lambda_\sigma^+$ be a finite subset which is contained in a
coset of $\Lambda(H)$. Then the subspace $L_A = \bigoplus_{\lambda \in A} W_\lambda \subset \c[G/P']$ is contained in some eigenspace of (right) action of $H$.
\end{enumerate}
\end{Prop}
\begin{proof}
Let $\lambda, \mu \in \supp(L)$ with $W_\lambda, W_\mu$ the corresponding eigenspaces
in $\c[G/P']$ for the right action of $P$.
Then the functions in $W_\lambda$ and $W_\mu$ are automatically eigenfunctions for the action of
$H$ with weights $i^*(\lambda)$ and $i^*(\mu)$ respectively,
where $i^*: \mathfrak{X}(P) \to \mathfrak{X}(H)$ is the restriction of characters.
On the other hand, every function in $L$ is an eigenfunction for $H$ with weight $\gamma$. This shows that $i^*(\lambda) = i^*(\mu) = \gamma$ and thus $i^*(\lambda - \mu) = 0$ i.e. $\lambda - \mu \in \Lambda(H)$. This proves 1). Now let $A \subset \Lambda_\sigma^+$
lay in a coset of $\Lambda(H)$. Then $i^*(A)$ consists of a single point $\gamma \in \mathfrak{X}(H)$.
Then $L_A = \bigoplus_{\lambda \in A} W_\lambda \subset \c[G/P']$ consists of eigenfunctions of $H$ with
weight $\gamma$. This finishes the proof of 2).
\end{proof}

From the previous proposition it follows that the moment polytope of $E$ lies in an affine subspace
parallel to $\Lambda_\r(H)$.

Consider the collection $\mathcal{S}(\Lambda(H))$ consisting of finite subsets $A$ of
$\Lambda_\sigma$ such that for any $\lambda, \mu \in A$
we have $\lambda - \mu \in \Lambda(H)$ (i.e. $A$ lies in a coset of $\Lambda(H)$).
Clearly $\mathcal{S}(\Lambda(H))$ is a semigroup under the addition of subsets. Also consider the collection $\mathcal{P}(\Lambda(H))$ of all the convex lattice polytopes $\Delta$ in $\Lambda_\sigma$ such that for any
two vertices $\lambda, \mu$ of $\Delta$ we have $\lambda - \mu \in \Lambda(H)$.
It is also clear that $\mathcal{P}(\Lambda(H))$ is a semigroup with respect to the
Minkowski sum of convex bodies.

As in Theorem \ref{th-finite-sets-convex-polytope-semigps} we have:
\begin{Prop} \label{prop-Grothendieck-semigp-S(Lambda-H)}
The map $A \mapsto \Delta(A)$, the convex hull of $A$, gives an isomorphism between the
Grothendieck semigroup of $\mathcal{S}(\Lambda(H))$ and the semigroup
$\mathcal{P}(\Lambda(H))$.
\end{Prop}

From Proposition \ref{prop-supp-parallel-Lambda-H} we get the following corollary:
\begin{Cor}
The map $L \mapsto \supp(L)$ gives an isomorphism between the semigroup
of finite dimensional subspaces of $\c[X]$ which are invariant under the left action of $G$ and
contained in some eigenspace for the right action of $H$, and the semigroup $\mathcal{S}(\Lambda(H))$ of finite subsets of $\Lambda_\sigma$ which lie in a coset of $\Lambda(H)$.
\end{Cor}

The next proposition follows immediately from Theorem \ref{th-G/P'} and Theorem \ref{th-E-L(E)}.
\begin{Prop} \label{prop-support-additive-G-lin-system}
\begin{enumerate}
\item Let $E_1, E_2 \in \tilde{\K}_G(Y)$ be two $G$-invariant linear systems. Then
$$\supp(E_1 E_2) = \supp(E_1) + \supp(E_2),$$ and hence
$$\Delta(E_1E_2) = \Delta(E_1) + \Delta(E_2).$$
\end{enumerate}
\end{Prop}

Now, as in Theorem \ref{th-semi-gp-inv-subspace-G/P'}
we obtain the following description for the semigroup of invariant linear
systems on $Y$ as well as a description of the completion of a linear system.
\begin{Cor}[Semigroup of $G$-invariant linear systems] \label{cor-semi-gp-inv-lin-system-G/H}
\begin{enumerate}
\item The map $E \mapsto \supp(E)$ gives an isomorphism between the semigroup $\tilde{\K}_G(Y)$ of $G$-invariant linear systems and the semigroup $\mathcal{S}(\Lambda(H))$ of finite subsets of $\Lambda_\sigma$ which lie in a coset of $\Lambda(H)$.

\item The map $E \mapsto \Delta(E)$ gives an isomorphism of the Grothendieck semigroup of $\tilde{\K}_G(Y)$ and the
semigroup $\mathcal{P}(\Lambda(H))$ of convex lattice polytopes in $\sigma_\r$ whose set of vertices lie in a coset of $\Lambda(H)$.

\item Let $\ol{E}$ be the completion of $E$. Then $\supp(\ol{E})$ is the intersection of
the coset of $\Lambda(H)$ containing $\supp(E)$ and the moment polytope $\Delta(E)$.
\end{enumerate}
\end{Cor}

As before let $\phi_\sigma$ be the homogeneous component of highest degree of the Weyl polynomial $F$
restricted to the linear span $\sigma_\r$ of the cone $\sigma$. Fix a Lebesgue measure on $\Lambda_\r(H)$ normalized with respect to $\Lambda(H)$, i.e.the smallest nonzero measure of a parallelepiped with vertices in
$\Lambda(H)$ is equal to $1$. We equip all the affine subspaces $a + \Lambda_\r(H)$, $a \in \sigma_\r$
with shifts of this Lebesgue measure.

\begin{Cor}[Intersection index of invariant linear systems] \label{cor-int-index-lin-system}
Let $E_1, \ldots, E_m \in \K_G(Y)$ be $G$-invariant linear systems where $m = \dim(Y)$. For
each $i$, let $\Delta_i = \Delta(E_i)$ be a moment polytope of $E_i$. We have
$$[E_1, \ldots, E_m] = m! I\phi_\sigma(\Delta_1, \ldots, \Delta_m),$$
where $I\phi_\sigma$ is the mixed integral of $\phi_\sigma$ for the polytopes
which are parallel to the linear space $\Lambda_\r(H)$.
\end{Cor}

\subsection{Intersection index as mixed volume} \label{sec-G-C}
In this section we rewrite the formula for the intersection index as
a mixed volume of certain polytopes (instead of mixed integral). To this end, we
use the so-called {\it Gelfand-Cetlin polytopes}.

In their classical paper \cite{G-C}, Gelfand and Cetlin constructed
a natural basis for any irreducible representation of $\GL(n, \c)$
and showed how to parameterize the elements of this basis with
integral points in a certain convex polytope. These polytopes are
called the {\it Gelfand-Cetlin polytopes}. Since then similar
constructions have been done for other classical groups and
analogous polytopes were defined (see \cite{B-Z1}). We will also
call them  {\it Gelfand-Cetlin polytopes} or for short {\it G-C
polytopes}. Consider the list of groups $\c^*$, $\SL(n_1, \c)$, $\SO(n_2,
\c)$ and $\SP(2n_3, \c)$, for any $n_1, n_2, n_3 \in \n$.
We say that $G$ is a {\it classical group}
if $G$ is in this list, or if $G$ can be constructed from the
groups in the list using the operations of taking direct product
and/or taking quotient by a finite central subgroup.
In this sense, the general linear group and the
orthogonal group are classical groups.

Let $G$ be a classical group. As usual let $d = \dim(G)$.
We have:
\begin{Th}[G-C polytopes] \label{th-G-C-additive}
For any classical group $G$ and for any $\lambda \in \Lambda^+$ one
can explicitly construct a polytope $\Delta_{GC}(\lambda) \subset
\r^{(d-r)/2}$, called the Gelfand-Cetlin polytope of $\lambda$, with
the following properties:
\begin{enumerate}
\item If $\lambda$ is integral then the dimension of $V_\lambda$ is equal to the number of
integral points in $\Delta_{GC}(\lambda)$.
\item The map $\lambda \mapsto \Delta_{GC}(\lambda)$ is linear, i.e.
for any two $\lambda, \gamma \in \Lambda^+_\r$ and $c_1, c_2 \geq 0$ we
have $= \Delta_{GC}(c_1\lambda + c_2\gamma) =
c_1\Delta_{GC}(\lambda) + c_2\Delta_{GC}(\gamma).$
\end{enumerate}
\end{Th}
The part (2) in the above theorem is an immediate corollary of the defining inequalities of
the G-C polytopes for the classical groups.

\begin{Def}[Newton polytope] \label{def-Newton-polytope}
Let $A$ be a finite nonempty set of dominant weights in $\Lambda_\sigma^+$.
Define the polytope $\tilde{\Delta}(A) \subset \sigma \times \r^{(d-r)/2}$
by:
$$\tilde{\Delta}(A) = \bigcup_{\lambda \in \Delta(A)} \{(\lambda, x) \mid x \in \Delta_{GC}(\lambda) \}.$$
In other words, the projection on the first factor maps
$\tilde{\Delta}(A)$ to the moment polytope $\Delta(A)$ and the
fibre over each $\lambda$ is the G-C polytope
$\Delta_{GC}(\lambda)$.
For a $G$-invariant subspace $L \subset \c[X]$ we define its {\it Newton polytope} $\tilde{\Delta}(L)$ to be
the polytope $\tilde{\Delta}(\supp(L))$. Similarly we define a {\it Newton polytope} $\tilde{\Delta}(E)$ of a $G$-invariant linear system $E$ to be the Newton polytope of $\supp(E)$.
\end{Def}

\begin{Cor}[Intersection index of subspaces as mixed volume] \label{cor-int-index-subspace-mixed-vol}
Let $L_1, \ldots, L_p \in \K_G(X)$ be $G$-invariant subspaces where $p = \dim(X)$. For
each $i$, let $\tilde{\Delta}_i = \tilde{\Delta}(L_i)$ be the Newton polytope of the subspace $L_i$. We have
$$[L_1, \ldots, L_p] = p! V(\tilde{\Delta}_1, \ldots, \tilde{\Delta}_p),$$
where $V$ denotes the mixed volume of convex bodies in the cone $\sigma$.
\end{Cor}

\begin{Cor}[Intersection index of linear systems as mixed volume] \label{cor-int-index-lin-system-mixed-vol}
Let $E_1, \ldots, E_m \in \tilde{\K}_G(Y)$ be $G$-invariant linear systems where $m = \dim(Y)$. For
each $i$, let $\tilde{\Delta}_i = \tilde{\Delta}(E_i)$ be a Newton polytope of $E_i$. We have
$$[E_1, \ldots, E_m] = m! V_H(\tilde{\Delta}_1, \ldots, \tilde{\Delta}_m),$$
where $V_H$ denotes the mixed volume of convex polytopes in $\sigma$ parallel to $\Lambda_\r(H)$.
\end{Cor}

\subsection{Case of $\GL(n, \c)$} \label{sec-GL(n)}
Let $G = \GL(n, \c)$ and let $B$ be the Borel subgroup of upper triangular matrices. Then the subgroup $T$
of diagonal matrices is a maximal torus, and the subgroup of upper-triangular matrices with $1$'s on the diagonal
is the maximal unipotent subgroup contained in $B$.
We identify the weight lattice $\Lambda$ with $\z^n$ and its linear span $\Lambda_\r$ with $\r^n$.
The Weyl group of $G$ is identified with the symmetric group $S_n$ which acts on $\r^n$ by permuting the coordinates.
The positive Weyl chamber for the choice of $B$ is
$\Lambda_\r^+ = \{ \lambda = (\lambda_1, \ldots, \lambda_n) \mid \lambda_n \geq \cdots \geq \lambda_1 \}$.

There is a one-to-one correspondence between the subsets $I = \{k_1 < \cdots < k_s\}$
of $\{1, \ldots, n-1\}$ and the
faces $$\sigma_I = \{ \lambda = (\lambda_1, \ldots, \lambda_n) \in \Lambda_\r^+ \mid
\lambda_{k_i+1} = \lambda_{k_i},~ i=1, \ldots, s\}$$
of the positive Weyl chamber. Also each subset $I$ then
corresponds to a parabolic subgroup $P_I$ consisting of the block upper-triangular matrices with
blocks of fixed sizes $k_1, k_2-k_1, \ldots, k_s-k_{s-1}, n-k_s$. One verifies that the commutator
subgroup $P_I'$ consists of block upper-triangular matrices where determinant of each block is equal to $1$.

Moreover, the torus $S = P_I/P_I'$ can be identified with $(\c^*)^s$ and
the natural map $P_I \to P_I/P_I'$ is given by $x \mapsto (\det(B_1), \ldots, \det(B_s))$,
where $x \in P$ and $B_1, \ldots, B_s$, are the blocks of sizes $k_1, k_2-k_1, \ldots, k_s-k_{s-1}, n-k_s$
respectively sitting on the diagonal of $x$.

Let us see that $G/P_I'$ is quasi-affine by giving an embedding of this variety in
some affine space. Let $g \in G$ be an invertible matrix with columns $C_1, \ldots, C_n$.
Consider the map $\Psi: G \to \bigoplus_{i=1}^s \bigwedge^{k_i} \c^n$ given by
$$g \mapsto \sum_{i=1}^s(C_1 \wedge \cdots \wedge C_{k_i}).$$
One verifies that $\Psi$ induces an embedding of $G/P_I'$ into the affine space $\bigoplus_{i=1}^s \bigwedge^{k_i} \c^n$.
This map is closely related to the generalized Pl\"{u}cker embedding.

Since $G$ is a Zariski open subset of $\c^{(n^2)}$, it is clear that $\Pic(G) = \{0\}$.

In the next example we consider a special case of a horospherical homogeneous space
which is related to the classical Bezout theorem.
\begin{Ex}
Let $G = \GL(n, \c)$ act on $\c^n$ in the natural way. Then $\c^n \setminus\{0\}$ is an orbit $O$.
Let $H$ be the $G$-stabilizer of $e_1$, the first standard basis vector. It contains the
subgroup of upper triangular matrices with $1$'s on the diagonal, i.e. $H$ is a horospherical subgroup and $O$ is a
horospherical homogeneous space $G/H$. If $n>1$, the space of regular functions $\c[O]$ is isomorphic to
the polynomial algebra on $\c^n$ and as a
$G$-module it decomposes into $$\c[O] = \bigoplus_{k=1}^\infty V_k,$$
where $V_k$ is the space of homogeneous polynomials of degree $k$ on $\c^n$. For each $k \geq 0$,
$V_k$ is an irreducible representation with highest weight $(k \geq 0 = \cdots = 0)$ (under the identification of
the dominant weights of $\GL(n, \c)$ with non-increasing sequences of integers
$\lambda = (\lambda_1 \geq \cdots \geq \lambda_n)$).
Let $F(k) = \dim(V_k) = $ number of monomials in $n$ variables and of total degree $k$. One knows that
$F(k) = {k+n-1 \choose n-1}$. Then $\phi(k) = (1/(n-1)!)k^{n-1}$ is the homogenous component of $F$ of highest degree.
For each finite set $A=\{a_1, \ldots, a_s\} \subset \z_{\geq 0}$ let $L_A$ be the space of
polynomials on $\c^n$ whose homogeneous components have degrees $a_1, \ldots, a_s$, i.e.
$\supp(L_A) = A$. Let $A_1, \ldots, A_n \subset \z_{\geq 0}$ be finite subsets and for each $i$, let $f_i \in L_i$ be a
generic polynomial. Then by Corollary \ref{cor-int-index-subspace}
the number of solutions of the system $f_1(x) = \cdots = f_n(x) = 0$ in
$\c^n \setminus \{0\}$ is equal to $n! I \phi (I_1, \ldots, I_n)$ where
for each $i$, $I_i = \Delta(A_i)$ is the interval which is the convex hull of the finite set $A_i$.
For each $k \geq 0$ let $\Delta_k$ be the G-C polytope associated to the dominant weight
$(k \geq 0 = \cdots = 0)$. From the defining equations of G-C polytopes we see that
$\Delta_k = \{ (x_1, \ldots, x_n) \in \r^{n-1} \mid k \geq x_{n-1} \geq \cdots \geq x_1 \geq 0\}$.
Now for a finite subset $A \subset \z_{\geq 0}$, let $I = \Delta(A) = [a, b]$.
Then the Newton polytope $\tilde{\Delta}(A) \subset \r^n$ is defined by
$$\tilde{\Delta}(A) = \{ (k, x_1, \ldots, x_{n-1}) \mid k \geq x_{n-1} \geq \cdots \geq x_1 \geq 0,~ a \geq k \geq b\}.$$
Let $\tilde{\Delta}_i$ denote the Newton polytope of the finite subset $A_i$. Then by
Corollary \ref{cor-int-index-subspace-mixed-vol} the
number of solutions of a generic system $f_1(x) = \cdots = f_n(x) = 0$ in $\c^n \setminus \{0\}$, where
$f_i \in L_i$, is given by $n! V(\tilde{\Delta}_1, \ldots, \tilde{\Delta}_n)$. Here $V$ denotes the mixed volume of
convex bodies in $\r^n$.
\end{Ex}

The last example concerns the degree of equivariant line bundles on partial flag varieties.
\begin{Ex}
Let $I \subset \{1, \ldots, n-1\}$ and let $H=P_I$ be the corresponding parabolic subgroup.
Clearly $H$ is a horospherical subgroup. Put $\dim(G/P_I) = m$.
For a dominant weight $\lambda \in \Lambda_{\sigma_I}^+$ let $\L_\lambda$ be the
corresponding $G$-linearized line bundle on $G/P_I$ and $E_\lambda
= H^0(G/P_I, \L_\lambda) \cong V_\lambda^*$ the corresponding complete linear system.
One sees that $\Delta(E_\lambda) = \lambda$ and $\tilde{\Delta}(E_\lambda) = \Delta_{GC}(\lambda)$.
Let $\lambda_1, \ldots, \lambda_m \in \Lambda_{\sigma_I}^+$ be dominant weights
with the corresponding linear systems $E_1, \ldots, E_m$. Then the intersection number of these
linear systems is given by $$[E_1, \ldots, E_m] = m!V(\Delta_{GC}(\lambda_1), \ldots, \Delta_{GC}(\lambda_m)),$$
where $V$ is the mixed volume of convex bodies in $\r^m$.
\end{Ex}

\vspace{.2cm}
{\small
\noindent Askold G. Khovanskii\\Department of
Mathematics, University of Toronto, Toronto, Canada;
Moscow Independent University; Institute for Systems Analysis, Russian Academy of Sciences.
{\it Email:} {\sf askold@math.utoronto.ca}\\

\noindent Kiumars Kaveh\\Department of Mathematics and Statistics, McMaster University, Hamilton, Canada.
{\it Email:} {\sf kavehk@math.mcmaster.ca}\\
}

\end{document}